\documentclass{amsart}
\usepackage{amsmath,amsthm,amssymb}
\usepackage{graphicx}
\usepackage{here}
\usepackage{morefloats}
\usepackage{afterpage}

\newtheorem{theorem}{Theorem}[section]
\newtheorem{proposition}[theorem]{Proposition}
\newtheorem{lemma}[theorem]{Lemma}
\newtheorem{corollary}[theorem]{Corollary}

\theoremstyle{definition}
\newtheorem{definition}[theorem]{Definition}
\newtheorem{remark}[theorem]{Remark}

\newtheorem{ack}{Acknowledgement}

\title[Knotting corks]{Knotting corks}
 
\author[AKBULUT and YASUI]{Selman Akbulut and Kouichi Yasui}
\thanks{The first author is partially supported by NSF, and the second author was partially supported by JSPS Research Fellowships for Young Scientists, and by GCOE, Kyoto University.}
\date{August 10, 2009}
\address{Department~of~Mathematics, Michigan State University, E.Lansing, MI, 48824, USA}
\email{akbulut@math.msu.edu}

\address{Research~Institute~for~Mathematical~Sciences, Kyoto~University, Kyoto 606-8502, Japan}\email{kyasui@kurims.kyoto-u.ac.jp}

\begin{document}

\begin{abstract}
It is known that every exotic smooth structure on a simply connected closed $4$-manifold is determined by a codimention zero compact contractible Stein submanifold and an involution on its boundary. Such a pair is called a cork. In this paper, we construct infinitely many knotted imbeddings of corks in $4$-manifolds such that they induce infinitely many different exotic smooth structures. We also show that we can imbed an arbitrary finite number of corks disjointly into 4-manifolds, so that the corresponding involutions on the boundary of the contractible $4$-manifolds give mutually different exotic structures. Furthermore, we construct similar examples for plugs. 
\end{abstract}

\maketitle


\section{Introduction}
In \cite{A1} the first author proved that $E(2)\#\overline{\mathbf{C}\mathbf{P}^2}$ changes its diffeomorphism type if we remove an imbedded copy of a Mazur manifold inside and reglue it by a natural involution on its bounday. This was later generalized to $E(n)\#\overline{\mathbf{C}\mathbf{P}^2}$ $(n\geq 2)$ by Bi\v zaca-Gompf~\cite{BG}. Here $E(n)$ denotes the relatively minimal elliptic surface with no multiple fibers and with Euler characteristic $12n$. Recently, the authors~\cite{AY1} and the first author~\cite{A5} constructed many such examples for other $4$-manifolds. 
The following general theorem was first proved independently by Matveyev~\cite{M}, Curtis-Freedman-Hsiang-Stong~\cite{C}, and later on strengthened by the first author and Matveyev~\cite{AM2}:

\begin{theorem}[\cite{M}, \cite{C},  \cite{AM2}]\label{th:1.1}
Let $X$ be a simply connected closed smooth $4$-manifold. If a smooth $4$-manifold $Y$ is homeomorphic but not diffeomorphic to $X$, then there exist a codimention zero contractible submanifold $C$ of $X$ and an involution $\tau$ on the boundary $\partial C$ such that $Y$ is obtained from $X$ by removing the submanifold $C$ and regluing it via the involution $\tau$. Such a pair $(C,\tau)$ is called a {\it Cork} of $X$. Furthermore, corks and their complements can always be made Stein manifolds. 
\end{theorem}

Hence smooth structures on simply connected closed $4$-manifolds are determined by corks. In this paper, subsequent to~\cite{AY1}, we explore the behavior of corks. 
It is a natural question whether every smooth structure on a $4$-manifold can be induced from a fixed cork $(C,\tau)$. In~\cite{AY1}, we showed that two different exotic smooth structures on a $4$-manifold can be obtained from the same cork (imbedded differently). Here, we prove that infinitely many different smooth structures on $4$-manifolds can be obtained from a fixed cork:

\begin{theorem}\label{th:knotting corks}
There exist a compact contractible Stein $4$-manifold $C$, an involution $\tau$ on the boundary $\partial C$, and infinitely many simply connected closed smooth $4$-manifolds $X_n$ $(n\geq 0)$ with the following properties:\smallskip \\
$(1)$ The $4$-manifolds $X_n$ $(n\geq 0)$ are mutually homeomorphic but not diffeomorphic;\smallskip \\
$(2)$ For each $n\geq 1$, the $4$-manifold $X_n$ is obtained from $X_0$ by removing a copy of $C$ and regluing it via $\tau$. Consequently, the pair $(C,\tau)$ is a cork of $X_0$. 

\vspace{.05in}
In particular, from $X_0$ we can produce infinitely many different smooth structures by the process of removing a copy of $C$ and regluing it via $\tau$. Consequently, these embeddings of $C$ into $X_0$ are mutually non-isotopic $($knotted copies of each other$)$. 
\end{theorem}

It is interesting to discuss positions of corks in $4$-manifolds. The next theorem says that we can put an arbitrary finite number of corks into mutually disjoint positions in $4$-manifolds:

\begin{theorem}\label{th:disjoint corks}
For each $n\geq 1$, there exist simply connected closed smooth $4$-manifolds $Y_{i}$ $(0\leq i\leq n)$, codimension zero compact contractible Stein submanifolds $C_i$ $(1\leq i\leq n)$ of $Y_0$, and an involution $\tau_i$ on the each boundary  $\partial C_i$ $(1\leq i\leq n)$ with the following properties:\smallskip \\
$(1)$ The submanifolds $C_i$ $(1\leq i\leq n)$ of $Y_{0}$ are mutually disjoint;\smallskip\\
$(2)$  $Y_{i}$ $(1\leq i\leq n)$ is obtained from $Y_{0}$ by removing the submanifold $C_i$ and  regluing it via $\tau_i$;\smallskip \\
$(3)$ The $4$-manifolds $Y_{i}$ $(0\leq i\leq n)$ are mutually homeomorphic but not diffeomorphic. In particular, the pairs $(C_i, \tau_i)$ $(1\leq i\leq n)$ are corks of $Y_{0}$.
\end{theorem}

The following theorem says that, for an embedding of a cork into a $4$-manifold, we can produce finitely many different cork structures of the $4$-manifold by only changing the involution of the cork without changing its embedding:

\begin{theorem}\label{th:involutions of corks}
For each $n\geq 1$, there exist simply connected closed smooth $4$-manifolds $Y_{i}$ $(0\leq i\leq n)$, an embedding of a compact contractible Stein $4$-manifold $C$ into $Y_{0}$, and involutions $\tau_i$ $(1\leq i\leq n)$ on the boundary $\partial C$ with the following properties:\smallskip \\
$(1)$ For each $1\leq i\leq n$, the $4$-manifold $Y_{i}$ is obtained from $Y_{0}$ by removing the submanifold $C$ and regluing it via $\tau_i$;\smallskip \\
$(2)$ The $4$-manifolds $Y_{i}$ $(0\leq i\leq n)$ are mutually homeomorphic but not diffeomorphic, hence the pairs $(C, \tau_i)$ $(1\leq i\leq n)$ are mutually different corks of $Y_{0}$.
\end{theorem}

In \cite{AY1}, we introduced new objects which we call \textit{Plugs}. We also construct similar examples for plugs (See Section~\ref{section:remark}). 

\begin{ack}
The second author would like to thank his adviser Hisaaki Endo for encouragement. 
\end{ack}

\section{Corks}
In this section, we recall corks. For details, see~\cite{AY1}.

\begin{definition}
Let $C$ be a compact contractible Stein $4$-manifold with boundary and $\tau: \partial C\to \partial C$ an involution on the boundary. 
We call $(C, \tau)$ a \textit{Cork} if $\tau$ extends to a self-homeomorphism of $C$, but cannot extend to any self-diffeomorphism of $C$. 
A cork $(C, \tau)$ is called a cork of a smooth $4$-manifold $X$, if  $C\subset X$ and $X$ changes its diffeomorphism type when we remove $C$ and reglue  it via $\tau$. Note that this operation does not change the homeomorphism type of $X$.
\end{definition}
\begin{remark}
In this paper, we always assume that corks are contractible. (We did not assume this in the more general definition of~\cite{AY1}.) Note that Freedman's theorem tells us that every self-diffeomorphism of the boundary of $C$ extends to a self-homeomorphism of $C$ when $C$ is a compact contractible smooth $4$-manifold. 
\end{remark}

\begin{definition}Let $W_n$ be the contractible smooth 4-manifold shown in Figure~$\ref{fig1}$. Let $f_n:\partial W_n\to \partial W_n$ be the obvious involution obtained by first surgering $S^1\times B^3$ to $B^2\times S^2$ in the interior of $W_n$, then surgering the other imbedded $B^2\times S^2$ back to $S^1\times B^3$ (i.e. replacing the dot and ``0'' in Figure ~$\ref{fig1}$). Note that the diagram of $W_n$ comes from a symmetric link.
\begin{figure}[ht!]
\begin{center}
\includegraphics[width=1.1in]{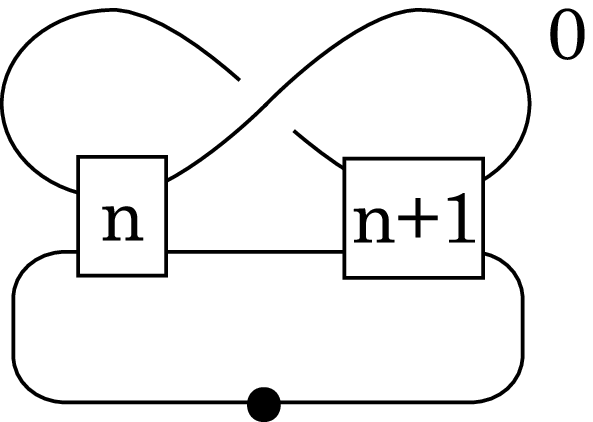}
\caption{$W_n$}
\label{fig1}
\end{center}
\end{figure}
\end{definition}
\vspace{-1.15\baselineskip }
\begin{theorem}[\cite{AY1}]\label{th:cork}
For $n\geq 1$, the pair $(W_n, f_n)$ is a cork. 
\end{theorem}

\section{Proof of Theorem~\ref{th:knotting corks}}
In this section, we prove Theorem~\ref{th:knotting corks} by using Fintushel-Stern's knot surgery and arguments similar to the proofs of~\cite[Theorem~3.4 and 3.5]{AY1}. \medskip 

First recall the following useful theorem from Gompf-Stipsicz~\cite[Section $9.3$]{GS}.

\begin{theorem}[Gompf-Stipsicz~\cite{GS}]\label{th:GS}
For $n\geq 1$, the elliptic surface $E(n)$ has the handle decomposition in Figure~$\ref{fig2}$. The obvious cusp neighborhood $($i.e. the dotted circle, one of the $-1$ framed meridians of the dotted circle, and the left most $0$-framed unknot$)$ is isotopic to the regular neighborhood of a cusp fiber of $E(n)$.%
\begin{figure}[ht!]
\begin{center}
\includegraphics[width=3.3in]{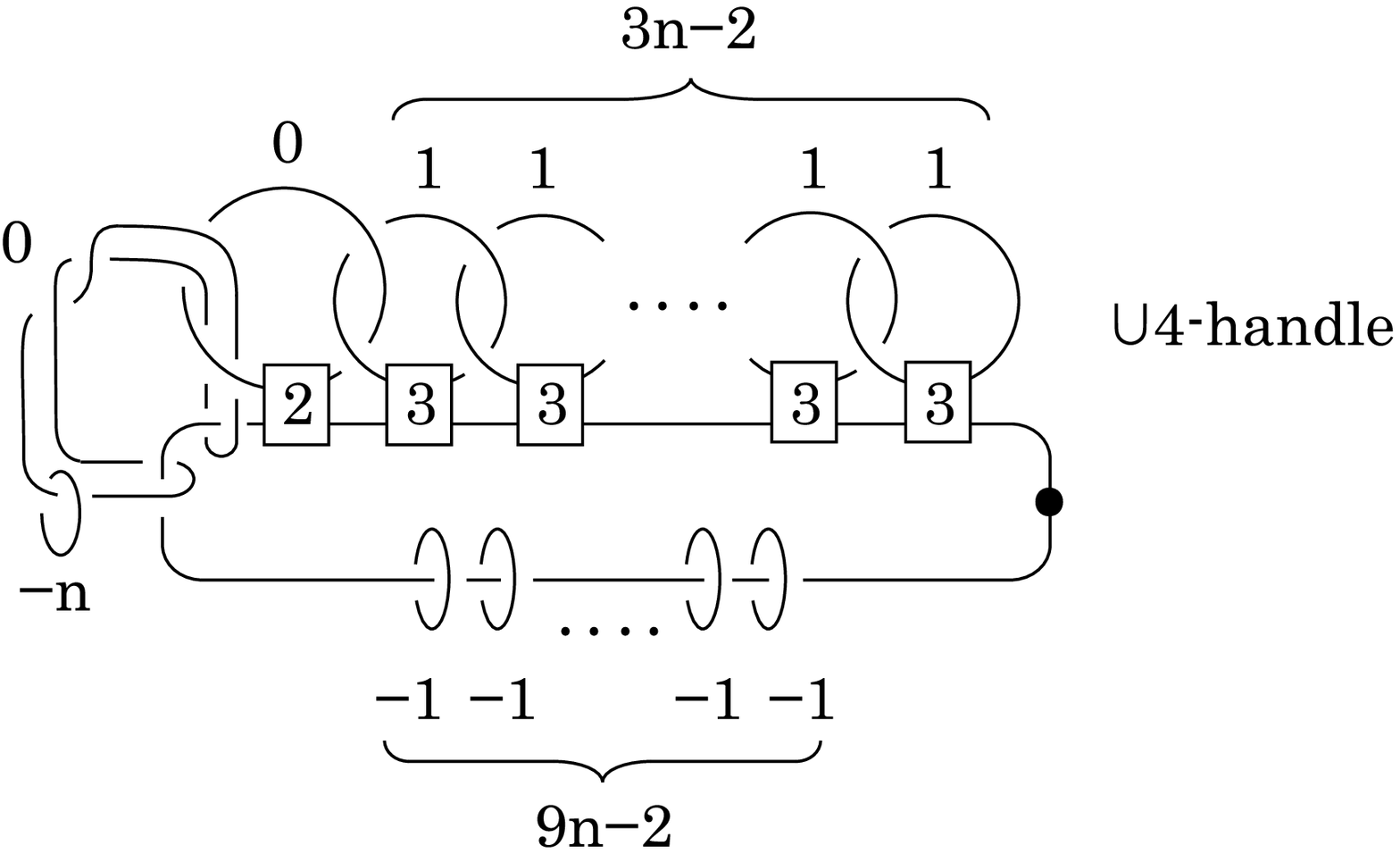}
\caption{$E(n)$}
\label{fig2}
\end{center}
\end{figure}
\end{theorem}

\begin{corollary}\label{cor1:GS}
For $n\geq 1$, the elliptic surface $E(n)$ has a handle decomposition as in Figure~$\ref{fig3}$. The obvious cusp neighborhood $($i.e. $0$-framed trefoil knot$)$ is isotopic to the regular neighborhood of a cusp fiber of $E(n)$.%
\begin{figure}[ht!]
\begin{center}
\includegraphics[width=3.5in]{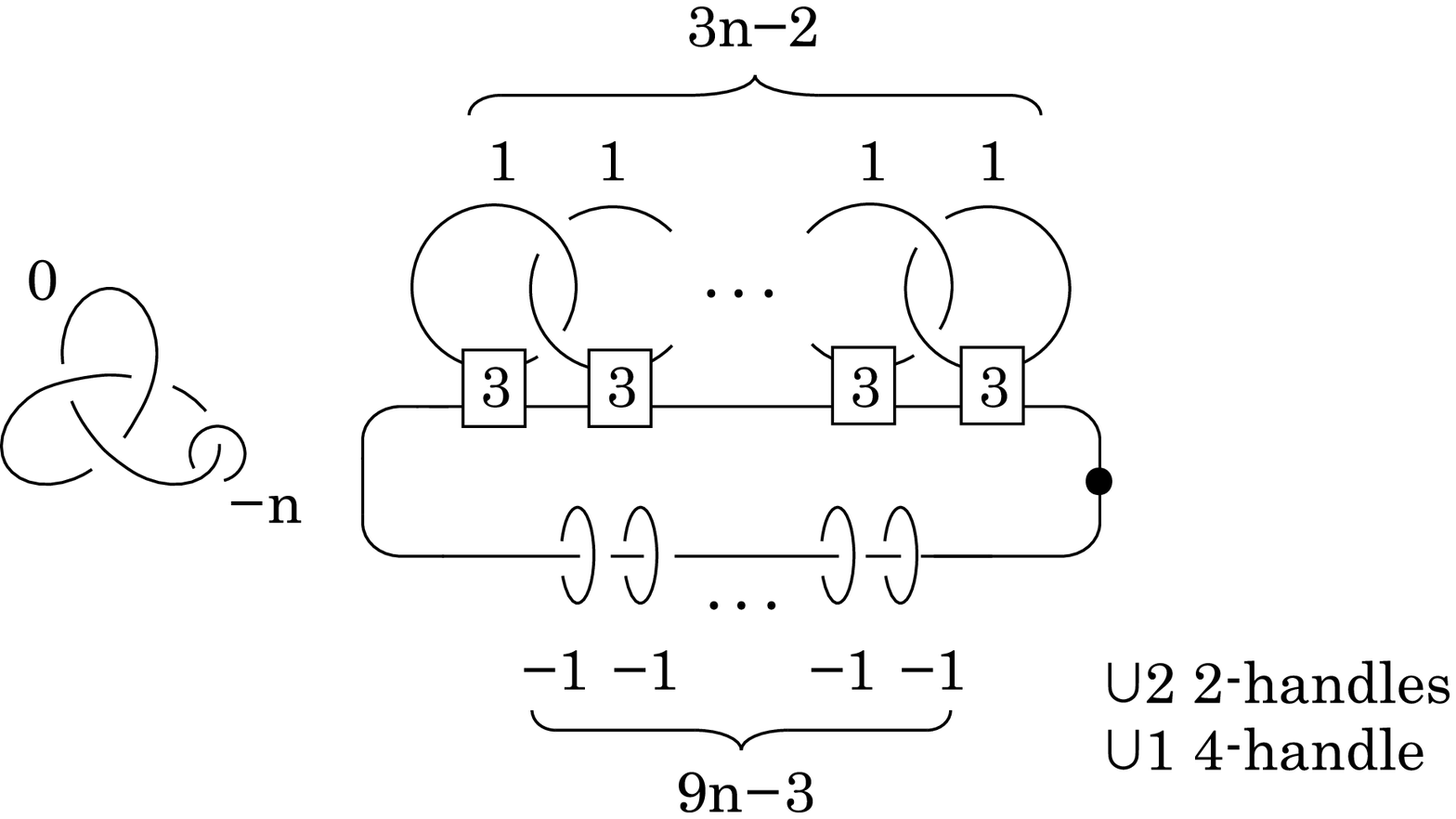}
\caption{$E(n)$}
\label{fig3}
\end{center}
\end{figure}
\end{corollary}
\begin{proof}
In Figure~$\ref{fig2}$, pull off the leftmost 0-framed unknot from the dotted circle by sliding over $-1$-framed knot. 
\end{proof}

\begin{corollary}\label{cor2:GS}
For each $p_1,p_2,\dots,p_n\geq 2$, the elliptic surface $E(p_1+p_2+\dots+p_n)$ has a handle decomposition as in Figure~$\ref{fig4}$. The obvious cusp neighborhood is isotopic to the regular neighborhood of a cusp fiber of $E(p_1+p_2+\dots+p_n)$. Here $k=9(p_1+p_2+\dots+p_n)-5n-4$.%
\begin{figure}[ht!]
\begin{center}
\includegraphics[width=3.7in]{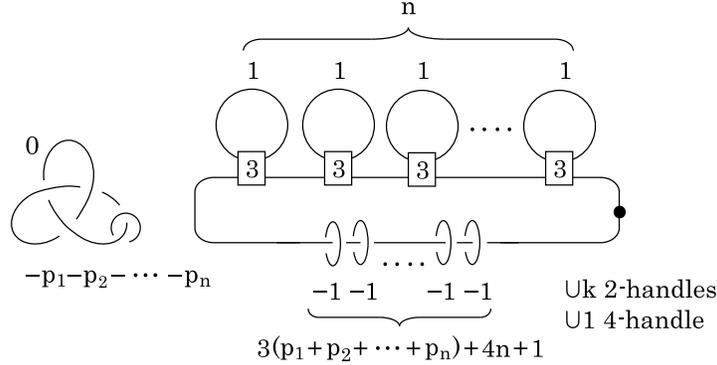}
\caption{$E(p_1+p_2+\dots+p_n)$ $(p_1,p_2,\dots,p_n\geq 2)$}
\label{fig4}
\end{center}
\end{figure}
\end{corollary}

\begin{remark}In this section, we do not use Corollary~\ref{cor2:GS}, it will be used in Section~\ref{sec:construction}.
\end{remark}

We can now easily get the following proposition.
 
\begin{proposition}\label{prop:cork}
The $4$-manifold $E(n)\# \overline{\mathbf{C}\mathbf{P}^2}$ $(n\geq 2)$ has a handle decomposition as in Figure~\ref{fig5}. The obvious cusp neighborhood in the figure is isotopic to the regular neighborhood of a cusp fiber of $E(n)$. Note that the figure contains $W_1$. 
\begin{figure}[ht!]
\begin{center}
\includegraphics[width=3.6in]{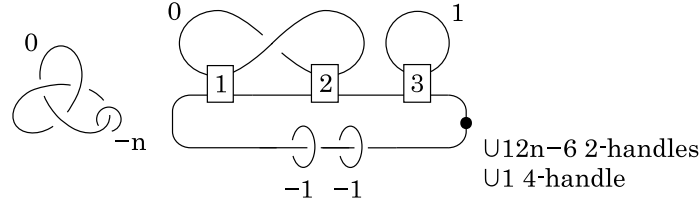}
\caption{$E(n)\# \overline{\mathbf{C}\mathbf{P}^2}$ $(n\geq 2)$}
\label{fig5}
\end{center}
\end{figure}
\end{proposition}
\vspace{-1.5\baselineskip}
\begin{proof}
It follows from Corollary~\ref{cor1:GS} that the submanifold of $E(n)$ in the first diagram of Figure~\ref{fig13} is disjoint from the cusp neighborhood of $E(n)$. Slide handles as in Figure~\ref{fig13}. Then we clearly get Figure~\ref{fig5} by blowing up. 
\end{proof}

\begin{definition}
$(1)$ Let $X$ be the smooth $4$-manifold obtained from $E(n)\# \overline{\mathbf{C}\mathbf{P}^2}$ $(n\geq 2)$ by removing the copy of $W_1$ in Figure~\ref{fig5} and regluing it via $f_1$. Note that $X$ contains a cusp neighborhood because the copy of $W_1$ in $E(n)\# \overline{\mathbf{C}\mathbf{P}^2}$ is disjoint from the cusp neighborhood in Figure~\ref{fig5}.\smallskip \\
$(2)$ Let $K$ be a knot in $S^3$ and $X_K$ the Fintushel-Stern's knot surgery (\cite{FS2}) with $K$ in the cusp neighborhood of $X$. \smallskip \\
$(3)$ Let $E(n)_K$ be the knot surgery with $K$ in the cusp neighborhood of $E(n)$. 
\end{definition}\smallskip 
The following corollary clearly follows from Proposition~\ref{prop:cork} and the definitions above. 
\begin{corollary}\label{cor:1_X_K}
$(1)$ The $4$-manifold $X$ splits off $\mathbf{CP}^2\# 2\overline{\mathbf{C}\mathbf{P}^2}$ as a connected summand. Furthermore, the cusp neighborhood of $X$ is disjoint from $\mathbf{CP}^2\# 2\overline{\mathbf{C}\mathbf{P}^2}$ in this connected sum decomposition. \smallskip \\
$(2)$ The $4$-manifold $X_K$ contains a copy of $W_1$ such that ${E(n)}_K\# \overline{\mathbf{C}\mathbf{P}^2}$ is obtained from $X_K$ by removing the copy of $W_1$ and regluing it via $f_1$.
\end{corollary}

\begin{corollary}\label{cor:2_X_K}
$(1)$ The $4$-manifold $X$ splits off $S^2\times S^2$ as a connected summand. Furthermore, the cusp neighborhood of $X$ is disjoint from $S^2\times S^2$ in this connected sum decomposition.  Consequently, the Seiberg-Witten invariant of $X$ vanishes. \smallskip \\
$(2)$ The $4$-manifold $X_K$ is diffeomorphic to $X$. In particular, the Seiberg-Witten invariant of $X_K$ vanishes. \smallskip \\
$(3)$ For each knot $K$ in $S^3$, there exists a copy of $W_1$ in $X$ such that ${E(n)}_K\# \overline{\mathbf{C}\mathbf{P}^2}$ is obtained from $X$ by removing the copy of $W_1$ and regluing it via $f_1$.\smallskip \\
$(4)$ If a knot $K$ in $S^3$ has the non-trivial Alexander polynomial, then ${E(n)}_K\# \overline{\mathbf{C}\mathbf{P}^2}$ is homeomorphic but not diffeomorphic to $X$,  in particular $(W_1,f_1)$ is a cork of $X$. 
\end{corollary}

\begin{proof}
The claim (1) easily follows from Corollary~\ref{cor:1_X_K}.(1) and the fact that, for every non-spin $4$-manifold $Y$, the $4$-manifold $Y\#\mathbf{CP}^2\# \overline{\mathbf{C}\mathbf{P}^2}$ is diffeomorphic to $Y\#(S^2\times S^2)$. 
The claim (1) and the definition of $X_K$ together with the stabilization theorem of knot surgery by the first author~\cite{A4} and Auckly~\cite{Au} show the claim (2). The claim (3) thus follows from Corollary~\ref{cor:1_X_K}.(2). Since the Seiberg-Witten invariant of ${E(n)}_K\# \overline{\mathbf{C}\mathbf{P}^2}$ does not vanish (Fintushel-Stern~\cite{FS2}), the claim (4) follows from the claim (1).
\end{proof}

Now we can easily prove Theorem~\ref{th:knotting corks}. 

\begin{proof}[Proof of Theorem~$\ref{th:knotting corks}$]
Let $X_0:=X$, and $K_i$ $(i\geq 1)$ be knots in $S^3$ with mutually different non-trivial Alexander polynomials. Define $X_i = {E(n)}_{K_i}\# \overline{\mathbf{C}\mathbf{P}^2}$. Then the  claim easily follows from Cororally~\ref{cor:2_X_K} and the Fintushel-Stern's formula (\cite{FS2}) of the Seiberg-Witten invariant of $E(n)_{K_i}$. 
\end{proof}

\begin{remark}\label{rem:knotted corks}
$(1)$ Freedman's theorem shows that $X_0$ is homeomorphic to $(2n-1)\mathbf{CP}^2\#10n \overline{\mathbf{C}\mathbf{P}^2}$. Since $X_0$ splits off $\mathbf{CP}^2\# \overline{\mathbf{C}\mathbf{P}^2}$, it is likely that $X_0$ is diffeomorphic to $(2n-1)\mathbf{CP}^2\#10n \overline{\mathbf{C}\mathbf{P}^2}$. \smallskip \\
$(2)$ The complement of the each copy of $W_1$ in $X_0 (\cong X_K)$ given in Corollary~\ref{cor:2_X_K}.(3) is simply connected. This claim easily follows from Figure~\ref{fig5} and the definition of $X_K$. Note that the knot surgered Gompf nuclei ${N_n}_K$ is simply connected. 
\smallskip \\
$(3)$ We proved Theorem~\ref{th:knotting corks} for the cork $(W_1,f_1)$. We can similarly prove Theorem~\ref{th:knotting corks} for many other corks, including $(W_n,f_n)$ and $(\overline{W}_n, \overline{f}_n)$. (For the definition of $(\overline{W}_n, \overline{f}_n)$, see~\cite{AY1}.) See also the proofs of~\cite[Proposition~3.3.(1) and (3)]{AY1}. 
\end{remark}

\section{Rational blowdown}
In this section we review the rational blowdown introduced 
by Fintushel-Stern \cite{FS1}. We also introduce a new relation between rational blowdowns and corks.\medskip 

Let $C_p$ and $B_p$ be the smooth $4$-manifolds defined by handlebody diagrams in Figure~\ref{fig6}, and $u_1,\dots,u_{p-1}$ elements of $H_2(C_p;\mathbf{Z})$ given by corresponding $2$-handles in the figure such that $u_i\cdot u_{i+1}=+1$ $(1\leq i\leq p-2)$.
The boundary $\partial C_p$ of $C_p$ is diffeomorphic to the lens space $L(p^2,p-1)$, and also diffeomorphic to the boundary $\partial B_p$ of $B_p$. 
\begin{figure}[ht!]
\begin{center}
\includegraphics[width=3.8in]{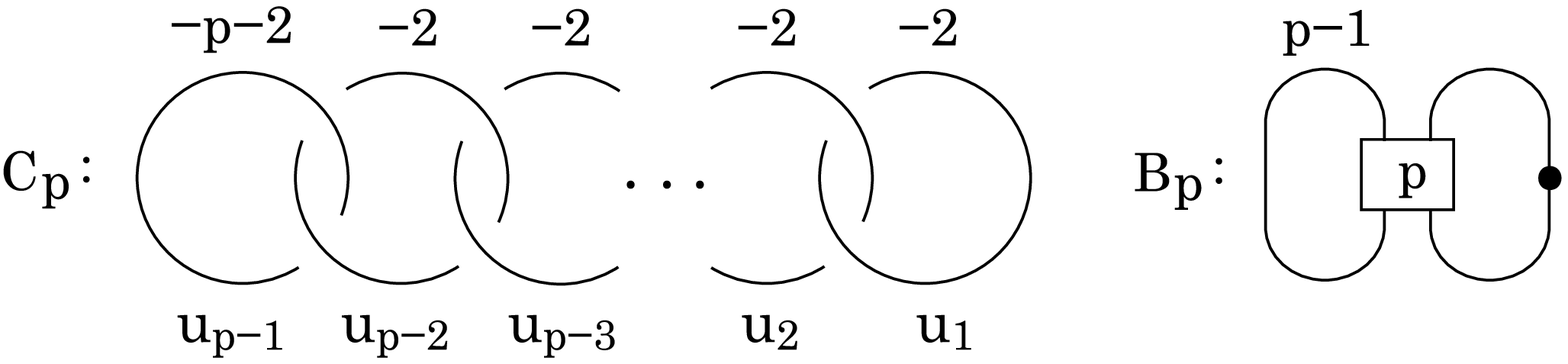}
\caption{}
\label{fig6}
\end{center}
\end{figure}

Suppose that $C_p$ embeds in a smooth $4$-manifold $Z$. 
Let $Z_{(p)}$ be the smooth $4$-manifold obtained from $Z$ by removing $C_p$ and gluing $B_p$ along the boundary. The smooth $4$-manifold $Z_{(p)}$ is called the rational blowdown of $Z$ along $C_p$. Note that $Z_{(p)}$ is uniquely determined up to diffeomorphism by a fixed pair $(Z,C_p)$ (see Fintushel-Stern~\cite{FS1}). 
This operation preserves $b_2^+$, decreases $b_2^-$, may create torsion in the first homology group. \medskip 

Rational blowdown has some relations with corks (\cite{AY1}). In this paper, we give the relation below, similarly to~\cite{AY1}. This relation is a key of our proof of Theorem~\ref{th:disjoint corks} and~\ref{th:involutions of corks}. 

\begin{theorem}\label{th:cork and rbd}Let $D_p$ be the smooth $4$-manifold in Figure~\ref{fig7} (notice $D_p$ is $C_p$ with two $2$-handles attached). 
Suppose that a smooth $4$-manifold $Z$ contains $D_{p}$. Let $Z_{(p)}$ be the rational blowdown of $Z$ along the copy of $C_p$ contained in $D_p$. 
Then the submanifold $D_p$ of $Z$ contains $W_{p-1}$ such that $Z_{(p)}\# (p-1)\overline{\mathbf{C}\mathbf{P}^2}$ is obtained from $Z$ by removing $W_{p-1}$ and regluing via $f_{p-1}$.
\begin{figure}[ht!]
\begin{center}
\includegraphics[width=2.15in]{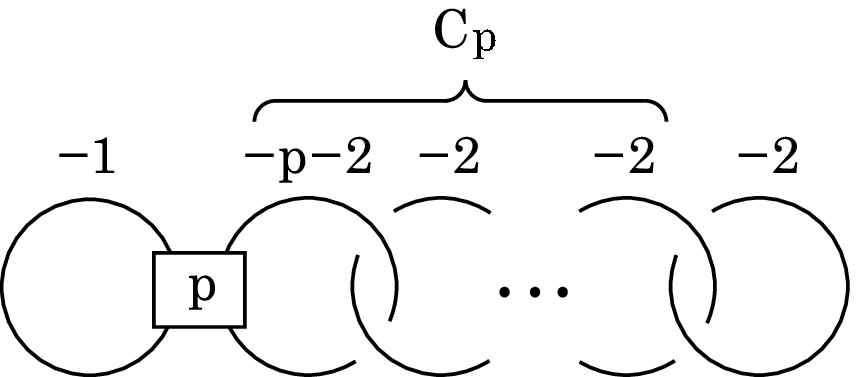}
\caption{$D_p$}
\label{fig7}
\end{center}
\end{figure}
\end{theorem}

\begin{proof}
We can easily get a handle decomposition of $D_p$ in the first diagram of Figure~$\ref{fig14}$, following the procedure in~\cite[Section~8.5]{GS}. (See also~\cite[Figure~14 and 15]{AY1}.) Slide handles as in Figure~$\ref{fig14}$. In the third diagram of Figure~$\ref{fig14}$, we can find a $0$-framed unknot which links the dotted circle geometrically once. Replace this dot and $0$ as in the first row of Figure~$\ref{fig15}$. This operation corresponds to removing $B^4$ in the submanifold $D_p$ of $Z$ and regluing $B^4$. Since every self-diffeomorphism of $S^3$ extends to a self-diffeomorphism of $B^4$, the operation above keeps the diffeomorphism type of $D_p$ and the embedding of $D_p$ into $X$ up to isotopy (In our situation, we can easily show this claim by checking that the operation above corresponds to canceling the 1-handle/2-handle pair and introducing a 1-handle/2-handle pair differently). As a consequence, we can easily find $W_{p-1}$ in $D_p$. Note that Figure~$\ref{fig16}$ is isotopic to the standard diagram of $W_{p-1}$. By removing $W_{p-1}$ in the submanifold $D_p$ of $Z$ and regluing it via $f_{p-1}$, we get the lower diagram of Figure~$\ref{fig15}$. 

The rational blowdown procedure in~\cite[Section~8.5]{GS} shows that $Z_{(p)}\# (p-1)\overline{\mathbf{C}\mathbf{P}^2}$ is obtained by replacing the dot and $0$ as in the left side of Figure~$\ref{fig15}$. Hence, we obtain $Z_{(p)}\# (p-1)\overline{\mathbf{C}\mathbf{P}^2}$ from $Z$ by removing $W_{p-1}$ and regluing it via $f_{p-1}$. 
\end{proof}

\section{Construction}\label{sec:construction}
In this section, we construct the examples of Theorem~\ref{th:disjoint corks} and~\ref{th:involutions of corks}, by imitating rational blowdown constructions in~\cite{Y1} and~\cite{Y2}. \medskip 

Let $T$ be the class of a regular fiber of $E(p_1+p_2+\dots+p_n)$ in $H_2(E(p_1+p_2+\dots+p_n);\mathbf{Z})$. Let $e_1,e_2,\dots,e_n$ be the standard basis of $H_2(n\overline{\mathbf{C}\mathbf{P}^2};\mathbf{Z})$ such that $e_i^2=-1$ $(1\leq i\leq n)$ and $e_i\cdot e_j=0$ $(i\neq j)$. 

\begin{proposition}\label{prop:disjoint C_p}
For each $n\geq 1$ and each $p_1,p_2,\dots,p_n\geq 2$, 
the $4$-manifold 
$E(p_1+p_2+\dots+p_n)\# n\overline{\mathbf{C}\mathbf{P}^2}$ admits a handle decomposition as in Figure~\ref{fig8}. The obvious cusp neighborhood in the figure is isotopic to the regular neighborhood of a cusp fiber of $E(p_1+p_2+\dots+p_n)$. The homology classes in the figure represent the homology classes given by corresponding $2$-handles. Here $k=11(p_1+p_2+\dots+p_n)+n-4$. 
\begin{figure}[ht!]
\begin{center}
\includegraphics[width=4.4in]{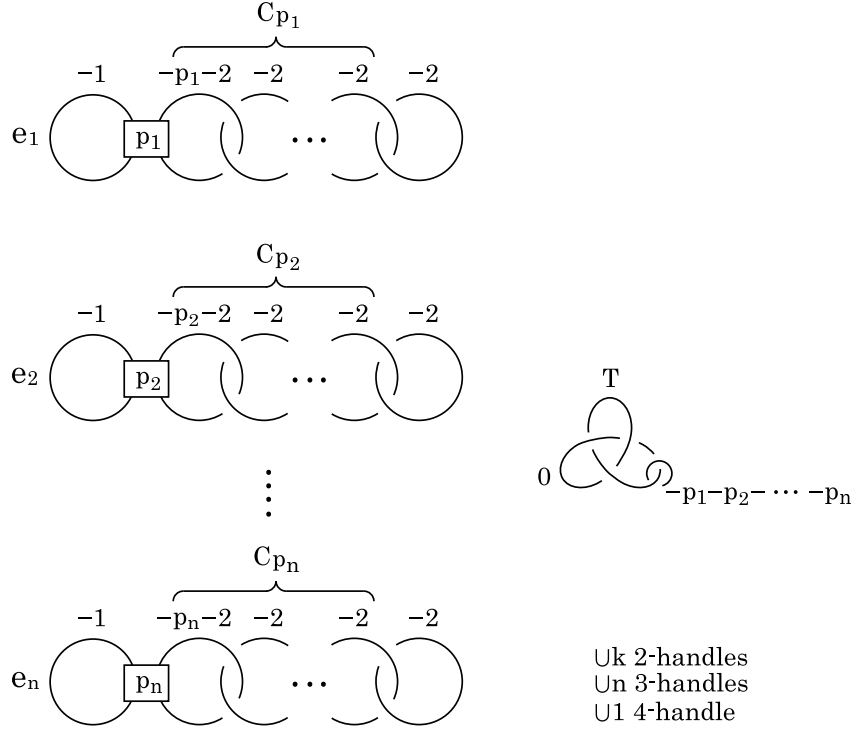}
\caption{$E(p_1+p_2+\dots+p_n)\# n\overline{\mathbf{C}\mathbf{P}^2}$ $(p_1,p_2,\dots,p_n\geq 2)$}
\label{fig8}
\end{center}
\end{figure}
\end{proposition}
\begin{proof}
We begin with the diagram of $E(p_1+p_2+\dots+p_n)$ in Figure~\ref{fig4}. Introduce a canceling $2$-handle/$3$-handle pair and slide handles as in Figure~\ref{fig17}. An isotopy gives the first diagram of Figure~\ref{fig18}. Slide handles and blow up as in Figure~\ref{fig18}. Handle slides give the first diagram of Figure~\ref{fig19}. Slide handles as in Figure~\ref{fig19}. We now have the diagram of $D_p$. By repeating this process and canceling the $1$-handle with a $-1$-framed 2-handle, we get Figure~\ref{fig8} of $E(p_1+p_2+\dots+p_n)\# n\overline{\mathbf{C}\mathbf{P}^2}$.
\end{proof}

\begin{definition}\label{def:disjoint cork}
$(1)$ Define $Y_0:=E(p_1+p_2+\dots+p_n)\# n\overline{\mathbf{C}\mathbf{P}^2}$. Let $Y_i'$ $(1\leq i\leq n)$ be the rational blowdown of $Y_0$ along the copy of $C_{p_i}$ in Figure~\ref{fig8}. Put $Y_i:=Y_i'\# (p_i-1)\overline{\mathbf{C}\mathbf{P}^2}$.\medskip \\
$(2)$ For $k_1,k_2,\dots,k_n\geq 1$, let $W(k_1,k_2,\dots,k_n)$ be the boundary sum $W_{k_1}\natural W_{k_2}\natural \cdots \natural W_{k_n}$. Figure~\ref{fig9} is a diagram of $W(k_1,k_2,\dots,k_n)$. Let $f^i(k_1,k_2,\dots,k_n)$ be the involution on the boundary $\partial W(k_1,k_2,\dots,k_n)$ obtained by replacing the dot and zero of the component of $W_{k_i}$.%
\begin{figure}[ht!]
\begin{center}
\includegraphics[width=4.3in]{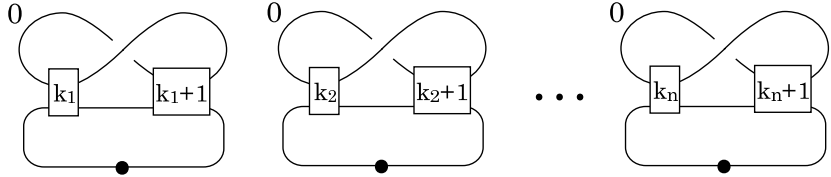}
\caption{$W(k_1,k_2,\dots,k_n)$}
\label{fig9}
\end{center}
\end{figure}
\end{definition}

\begin{lemma}
For each $k_1,k_2,\dots,k_n\geq 1$, the manifold $W(k_1,k_2,\dots,k_n)$ is a compact contractible Stein $4$-manifold.
\end{lemma}

\begin{proof}
We can check this by changing the $1$-handle notations of $W(k_1,k_2,\dots,k_n)$, and putting the 2-handles in Legendrian positions. (For such a diagram of $W_n$, see Figure~\ref{fig10}.) Now all we have to check is the Eliashberg criterion: the framings on the $2$-handles are less than Thurston-Bennequin number. 
\begin{figure}[ht!]
\begin{center}
\includegraphics[width=1.9in]{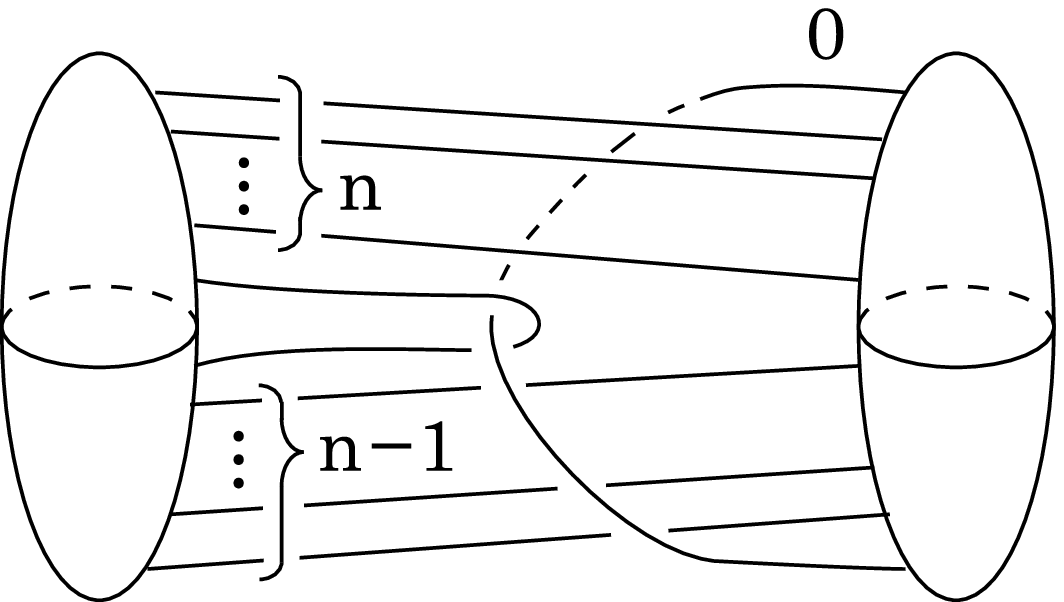}
\caption{$W_n$}
\label{fig10}
\end{center}
\end{figure}
\end{proof}

\begin{proposition}\label{prop:X_i}
$(1)$ The $4$-manifold $Y_0$ contains mutually disjoint copies of $W_{p_1}, W_{p_2}, \dots, W_{p_n}$ such that, for each $i$, the $4$-manifold $Y_i$ is obtained from $Y_0$ by removing the copy of $W_{p_i}$ and regluing it via the involution $f_{p_i}$.\medskip\\
$(2)$ The $4$-manifold $Y_0$ contains a fixed copy of $W(p_1-1,p_2-1,\dots,p_n-1)$ such that, for each $i$, the $4$-manifold $Y_i$ is obtained from $Y_0$ by removing the copy of $W(p_1-1,p_2-1,\dots,p_n-1)$ and regluing it via the involution $f^i(p_1-1,p_2-1,\dots,p_n-1)$.
\end{proposition}

\begin{proof}
Proposition~\ref{prop:disjoint C_p} and Theorem~\ref{th:cork and rbd} clearly show the claims $(1)$ and (2). 
\end{proof}

\section{Proof of Theorem~\ref{th:disjoint corks} and~\ref{th:involutions of corks}}
\subsection{Seiberg-Witten invariants}
In this subsection, we briefly review facts about the Seiberg-Witten invariants. 
For details and examples of computations, see, for example, Fintushel-Stern~\cite{FS3}.\medskip 

Suppose that $Z$ is a simply connected closed smooth $4$-manifold with $b_2^+(Z)>1$ and odd. Let $\mathcal{C}(Z)$ be the set of 
characteristic elements of $H^2(Z;\mathbf{Z})$. 
Then the Seiberg-Witten invariant $SW_{Z}: \mathcal{C}(Z)\to \mathbf{Z}$ is defined. 
Let $e(Z)$ and $\sigma(Z)$ be the Euler characteristic and the signature of $Z$, respectively, and $d_Z(K)$ the even integer defined by $d_Z(K)=\frac{1}{4}(K^2-2e(Z)-3\sigma(Z))$ 
for $K\in \mathcal{C}(Z)$. 
If $SW_Z(K)\neq 0$, then $K$ is called a Seiberg-Witten basic class of $Z$. We denote $\beta(Z)$ as the set of the Seiberg-Witten basic classes of $Z$. The following theorem is well-known.

\begin{theorem}[Witten~\cite{W}, cf.\ Gompf-Stipsicz~\cite{GS}]\label{th:SW of E(n)}For $n\geq 2$, \smallskip \\
$(1)$ $\beta(E(n))=\{k\cdot PD(T)\mid k\equiv 0\; (\textnormal{mod}\; 2), \lvert k\rvert \leq n-2 \}$$;$\smallskip \\
$(2)$ $\beta(E(n)\# m\overline{\mathbf{C}\mathbf{P}^2})=\{k\cdot PD(T)\pm E_1\pm E_2\pm \dots \pm E_m\mid k\equiv 0\; (\textnormal{mod}\; 2), \lvert k\rvert \leq n-2 \}$. Here $T$ denotes the class of a regular fiber of $E(n)$ in $H_2(E(n);\mathbf{Z})$, and $E_1,E_2,\dots,E_m$ denotes the standard basis of $H^2(m\overline{\mathbf{C}\mathbf{P}^2};\mathbf{Z})$. 
\end{theorem}
%

We here recall the change of the Seiberg-Witten invariants by rationally blowing down. Assume further that $Z$ contains a copy of $C_p$. Let $Z_{(p)}$ be the rational blowdown of $Z$ along the copy of $C_p$. Suppose that $Z_{(p)}$ is simply connected. The following theorems are obtained by Fintushel-Stern~\cite{FS1}.

\begin{theorem}[Fintushel-Stern \cite{FS1}]\label{th:SW1}
For every element $K$ of $\mathcal{C}(Z_{(p)})$, there exists an element $\tilde{K}$ of $\mathcal{C}(Z)$ such that 
$K\rvert _{Z_{(p)}-\text{\normalfont{int}}\,B_{p}}=\tilde{K}\rvert _{Z-\text{\normalfont{int}}\,C_{p}}$ and 
$d_{Z_{(p)}}(K)=d_Z(\tilde{K})$. Such an element $\tilde{K}$ of $\mathcal{C}(Z)$ is called a lift of $K$.
\end{theorem}

\begin{theorem}[Fintushel-Stern \cite{FS1}]\label{th:SW2}
If an element $\tilde{K}$ of $\mathcal{C}(Z)$ is a lift of some element $K$ of $\mathcal{C}(Z_{(p)})$, then $SW_{Z_{(p)}}(K)=SW_{Z}(\tilde{K})$. 
\end{theorem}

\begin{theorem}[Fintushel-Stern \cite{FS1}, cf.~Park \cite{P1}]\label{th:SW3}
If an element $\tilde{K}$ of $\mathcal{C}(Z)$ satisfies that $(\tilde{K}\rvert _{C_p})^2=1-p$ and 
$\tilde{K}\rvert _{\partial C_p}=mp\in \mathbf{Z}_{p^2}\cong H^2(\partial C_p;\mathbf{Z})$ 
with $m\equiv p-1\pmod 2$, then there exists an element $K$ of $\mathcal{C}(Z_{(p)})$ such that $\tilde{K}$ is a lift of $K$. 
\end{theorem}

\begin{corollary}\label{cor:SW4}
If an element $\tilde{K}$ of $\mathcal{C}(Z)$ satisfies $\tilde{K}(u_1)=\tilde{K}(u_2)=\dots=\tilde{K}(u_{p-2})=0$ and $\tilde{K}(u_{p-1})=\pm p$, then $\tilde{K}$ is a lift of some element $K$ of $\mathcal{C}(Z_{(p)})$. 
\end{corollary}

\subsection{Computation of SW invariants}In this subsection, we prove Theorem~\ref{th:disjoint corks} and~\ref{th:involutions of corks} by computing the Seiberg-Witten invariants of the $4$-manifolds $Y_i$ $(0\leq i\leq n)$ in Definition~\ref{def:disjoint cork}.\medskip 

For a smooth $4$-manifold $Z$ we denote $N(Z)$ as the number of elements of $\beta(Z)$. 

\begin{lemma}$N(Y_i)=2^{p_i-1}N(Y_0)$ $(1\leq i\leq n)$
\end{lemma}

\begin{proof}Proposition~\ref{prop:disjoint C_p}, Theorem~\ref{th:SW of E(n)} and Corollary~\ref{cor:SW4} show that every Seiberg-Witten basic class of $Y_0$ is a lift of some element of $\mathcal{C}(Y_i')$, and that these basic classes of $Y_0$ have mutually different restrictions to $Y_0-\text{\normalfont{int}}\,C_{p_i} (=Y_i'-\text{\normalfont{int}}B_{p_i})$. Note that every element of $H^2(Y_i';\mathbf{Z})$ is uniquely determined by its restriction to ${Y_i'-\text{\normalfont{int}}B_{p_i}}$. (We can easily check this from the cohomology exact sequence for the pair $(Y_i', Y_i'-\text{\normalfont{int}}B_{p_i})$.) Hence Theorems~\ref{th:SW1} and~\ref{th:SW2} give $N(Y_i')=N(Y_0)$. Now the required claim follows from the blow-up formula. 
\end{proof}

\begin{corollary}
If $p_1,p_2,\dots,p_n\geq 2$ are mutually different, then $Y_i$ $(0\leq i\leq n)$ are mutually homeomorphic but not diffeomorphic.  
\end{corollary}

\begin{proof}[Proof of Theorem~$\ref{th:disjoint corks}$ and~$\ref{th:involutions of corks}$] 
The theorems clearly follow from the corollary above and Proposition~\ref{prop:X_i}.
\end{proof}

\section{Further remarks}\label{section:remark}

\subsection{Plugs}\label{subsection:plug}
In~\cite{AY1}, we introduced new objects which we call \textit{Plugs}.

\begin{definition}
Let $P$ be a compact Stein $4$-manifold with boundary and $\tau: \partial P\to \partial P$ an involution on the boundary, which cannot extend to any self-homeomorphism of $P$. We call $(P, \tau)$ a \textit{Plug} of $X$,  if  $P\subset X$ and $X$ keeps its homeomorphism type and changes its diffeomorphism type when removing $P$ and gluing it via $\tau$. 
We call $(P, \tau)$ a \textit{Plug} if there exists a smooth $4$-manifold $X$ such that $(P, \tau)$ is a plug of $X$. 
\end{definition}

\begin{definition} Let $W_{m,n}$ be the smooth 4-manifold  in Figure~$\ref{fig11}$, and 
let $f_{m,n}:\partial W_{m,n}\to \partial W_{m,n}$ be the obvious involution obtained by first surgering $S^1\times B^3$ to $B^2\times S^2$ in the interior of $W_{m,n}$, then surgering the other imbedded $B^2\times S^2$ back to $S^1\times B^3$ (i.e. replacing the dots in Figure ~$\ref{fig1}$). Note that the diagram of $W_{m,n}$ is a symmetric link.
\begin{figure}[ht!]
\begin{center}
\includegraphics[width=0.8in]{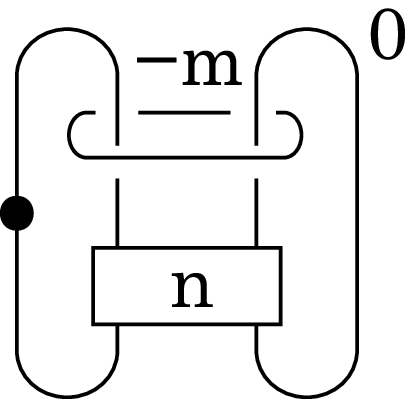}
\caption{}
\label{fig11}
\end{center}
\end{figure}
\end{definition}

\begin{theorem}[{\cite{AY1}}]\label{th:basic_plug}
For $m\geq 1$ and $n\geq 2$, the pair $(W_{m,n}, f_{m,n})$ is a plug.
\end{theorem}

As we pointed out in~\cite[Remark~5.3]{AY1}, removing and regluing plugs have naturally the same effect as cork operations, under some conditions. We can easily check that our constructions in this paper are such cases. Therefore we can similarly prove the theorems below. 

The following theorem shows that infinitely many different smooth structures on $4$-manifolds can be obtained from a fixed plug: 

\begin{theorem}\label{th:knotting plugs}
There exist a simply connected compact Stein $4$-manifold $P$, an involution $\tau$ on the boundary $\partial P$, and infinitely many simply connected closed smooth $4$-manifolds $X_n$ $(n\geq 0)$ with the following properties:\smallskip \\
$(1)$ The pair $(P,\tau)$ is a plug;\smallskip \\
$(2)$ The $4$-manifolds $X_n$ $(n\geq 0)$ are mutually homeomorphic but not diffeomorphic;\smallskip \\
$(3)$ For each $n\geq 1$, the $4$-manifold $X_n$ is obtained from $X_0$ by removing a copy of $P$ and regluing it via $\tau$. Consequently, the pair $(P,\tau)$ is a plug of $X_0$. 

\vspace{.05in}
In particular, from $X_0$ we can produce infinitely many different smooth structures by the process of removing a copy of $P$ and regluing it via $\tau$. Consequently, these embeddings of $P$ into $X_0$ are mutually non-isotopic $($knotted copies of each other$)$. 
\end{theorem}

\begin{proof}We begin with the proposition below. 
\begin{proposition}\label{prop:plug}
The $4$-manifold $E(n)\# \overline{\mathbf{C}\mathbf{P}^2}$ $(n\geq 2)$ has a handle decomposition as in Figure~\ref{fig12}. The obvious cusp neighborhood in the figure is isotopic to the regular neighborhood of a cusp fiber of $E(n)$. Note that the figure contains $W_1$ and $W_{1,2}$. Furthermore, this copy of $W_1$ is isotopic to the copy of $W_1$ obtained in Proposition~\ref{prop:cork}. 
\begin{figure}[ht!]
\begin{center}
\includegraphics[width=3.90in]{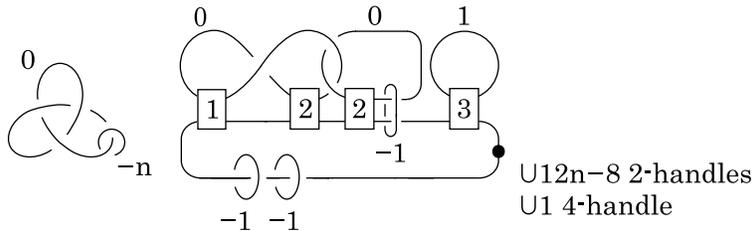}
\caption{$E(n)\# \overline{\mathbf{C}\mathbf{P}^2}$ $(n\geq 2)$}
\label{fig12}
\end{center}
\end{figure}
\end{proposition}

\begin{proof}[Proof of Proposition~$\ref{prop:plug}$]
It follows from Corollary~\ref{cor1:GS} that the submanifold of $E(n)$ in the first diagram of Figure~\ref{fig13} is disjoint from the cusp neighborhood of $E(n)$. Slide handles as in Figure~\ref{fig20}. Now we can easily check the required claim. See also the proof of Proposition~\ref{prop:cork}. 
\end{proof}

Let $X$ be the smooth $4$-manifold obtained from $E(n)\# \overline{\mathbf{C}\mathbf{P}^2}$ $(n\geq 2)$ by removing the copy of $W_{1,2}$ in Figure~\ref{fig12} and regluing it via $f_{1,2}$. Note that two $0$-framed $2$-handles of the copies of $W_1$ and $W_{1,2}$ in Figure~\ref{fig12} link geometrically once. Therefore, similarly to the proof of Theorem~\ref{th:cork and rbd}, we can easily show that $X$ is obtained from $E(n)\# \overline{\mathbf{C}\mathbf{P}^2}$ by removing the copy of $W_{1}$ in Figure~\ref{fig12} and regluing it via $f_{1}$. The rest of the proof of Theorem~\ref{th:knotting plugs} proceeds as in the proof of Theorem~\ref{th:knotting corks}. 
\end{proof}

\begin{remark}$(1)$ As we pointed out in the proof above, the cork operation and the plug operation are the same, for our example of Theorem~\ref{th:knotting corks} and \ref{th:knotting plugs}. \smallskip \\
$(2)$ In Theorem~\ref{th:knotting plugs}, we obtained infinitely many knotted embeddings of $W_{1,2}$ into $X$. Each embedding gives the same subspace of $H_2(X;\mathbf{Z})$ corresponding to $H_2(W_{1,2};\mathbf{Z})(\cong \mathbf{Z})$. Thus the construction in the theorem might give useful applications to find homologous but non-isotopic surfaces in $X$. \smallskip \\
$(3)$ We proved Theorem~\ref{th:knotting plugs} for the plug $(W_{1,2}, f_{1,2})$. We can similarly prove Theorem~\ref{th:knotting plugs} for many other plugs, including $(W_{m,n}, f_{m,n})$ $(m\geq 1,\, n\geq 2)$. 
\end{remark}

The next theorem says that we can put an arbitrary finite number of plugs into mutually disjoint positions in $4$-manifolds:
\begin{theorem}\label{th:disjoint plugs}
For each $n\geq 1$, there exist simply connected closed smooth $4$-manifolds $Y_{i}$ $(0\leq i\leq n)$, codimension zero simply connected compact Stein submanifolds $P_i$ $(1\leq i\leq n)$ of $Y_0$, and an involution $\tau_i$ on the each boundary $\partial P_i$ $(1\leq i\leq n)$ with the following properties:\smallskip \\
$(1)$ The pairs $(P_i,\tau_i)$ $(1\leq i\leq n)$ are plugs; \smallskip\\
$(2)$ The submanifolds $P_i$ $(1\leq i\leq n)$ of $Y_{0}$ are mutually disjoint;\smallskip\\
$(3)$ Each $Y_{i}$ $(1\leq i\leq n)$ is obtained from $Y_{0}$ by removing the submanifold $P_i$ and  regluing it via $\tau_i$;\smallskip \\
$(4)$ The $4$-manifolds $Y_{i}$ $(0\leq i\leq n)$ are mutually homeomorphic but not diffeomorphic. In particular, the pairs $(P_i, \tau_i)$ $(1\leq i\leq n)$ are plugs of $Y_{0}$.
\end{theorem}

The following theorem says that, for an embedding of a plug into a $4$-manifold, we can produce finitely many different plug structures of the $4$-manifold by only changing the involution of the plug without changing its embedding:

\begin{theorem}\label{th:involutions of plugs}
For each $n\geq 1$, there exist simply connected closed smooth $4$-manifolds $Y_{i}$ $(0\leq i\leq n)$, an embedding of a simply connected compact Stein $4$-manifold $P$ into $Y_{0}$, and involutions $\tau_i$ $(1\leq i\leq n)$ on the boundary $\partial P$ with the following properties:\smallskip \\
$(1)$ The pairs $(P,\tau_i)$ $(1\leq i\leq n)$ are plugs. \smallskip\\
$(2)$ For each $1\leq i\leq n$, the $4$-manifold $Y_{i}$ is obtained from $Y_{0}$ by removing the submanifold $P$ and regluing it via $\tau_i$;\smallskip \\
$(3)$ The $4$-manifolds $Y_{i}$ $(0\leq i\leq n)$ are mutually homeomorphic but not diffeomorphic, hence the pairs $(P, \tau_i)$ $(1\leq i\leq n)$ are mutually different plugs of $Y_{0}$.
\end{theorem}

\begin{proof}[Proof of Theorem~$\ref{th:disjoint plugs}$ and~$\ref{th:involutions of plugs}$]
According to the theorem below, we can show the required claims similarly to the proof of Theorem~$\ref{th:disjoint corks}$ and~$\ref{th:involutions of corks}$. (As for Theorem~$\ref{th:involutions of plugs}$, we also use the argument similar to~\cite[Lemma~2.7.(3)]{AY1}.)

\begin{theorem}[{\cite[Theorem~5.1.(3)]{AY1}}]\label{th:plug and rbd}Suppose that a smooth $4$-manifold $Z$ contains the $4$-manifold $D_{p}$ in Figure~\ref{fig7}. Let $Z_{(p)}$ be the rational blowdown of $Z$ along the copy of $C_p$ contained in $D_p$. 
Then the submanifold $D_p$ of $Z$ contains $W_{1,p}$ such that $Z_{(p)}\# (p-1)\overline{\mathbf{C}\mathbf{P}^2}$ is obtained from $Z$ by removing $W_{1,p}$ and regluing it via $f_{1,p}$.
\end{theorem}
\end{proof}
\subsection{Knotted contractible 4-manifolds} Lickorish~\cite{Lic} constructed large families of contractible 4-manifolds that have two non-isotopic embeddings into $S^4$. Livingston~\cite{Liv} later gave large families of contractible 4-manifolds that have infinitely many mutually non-isotopic embeddings into $S^4$. These embeddings are detected by the fundamental group of their complements. The corks $W_n$ can  be knotted  in $S^4$ with simply connected complements (even PL knotted), this is because  doubling $W_n$  both by  the identity and by the involution $f_n:\partial W_n \to \partial W_n $ give $S^4$, and $f_n$ takes a slice knot to a non-slice knot (cf. \cite{A6}). Obviously these knotted imbeddings $W_{n}\subset S^4$ are not corks of $S^4$. It is not known whether $S^4$ admits cork imbeddings (i.e. it is not known whether $S^4$ admits an exotic smooth structure). 
Theorem~\ref{th:knotting corks} of this paper gives infinitely many mutually non-isotopic embeddings of $W_1$ (and also other contractible $4$-manifolds [See Remark~\ref{rem:knotted corks}.(3)]) into the $4$-manifold $X_0$ with simply connected complements, so that the imbeddings give mutually different cork structures of $X_0$. 

\begin{figure}[ht!]
\begin{center}
\includegraphics[width=4.4in]{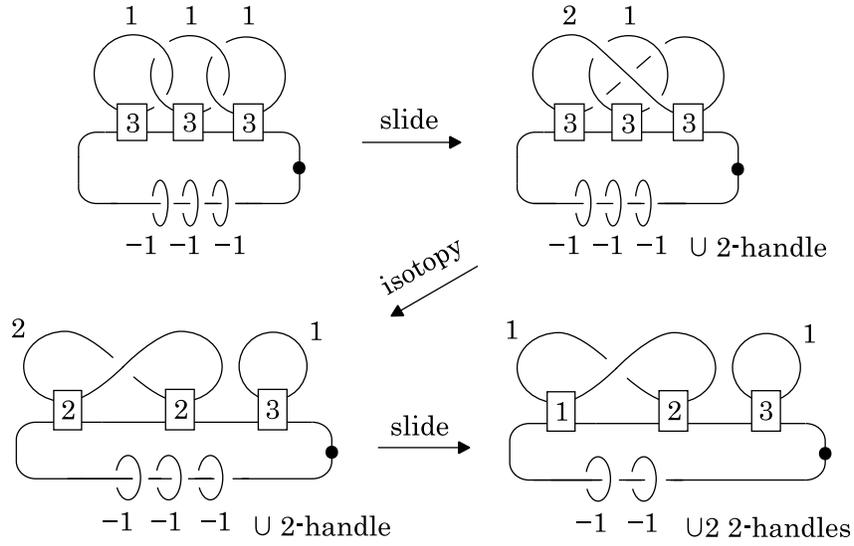}
\caption{handle slides}
\label{fig13}
\end{center}
\end{figure}
\begin{figure}[ht!]
\begin{center}
\includegraphics[width=4.7in]{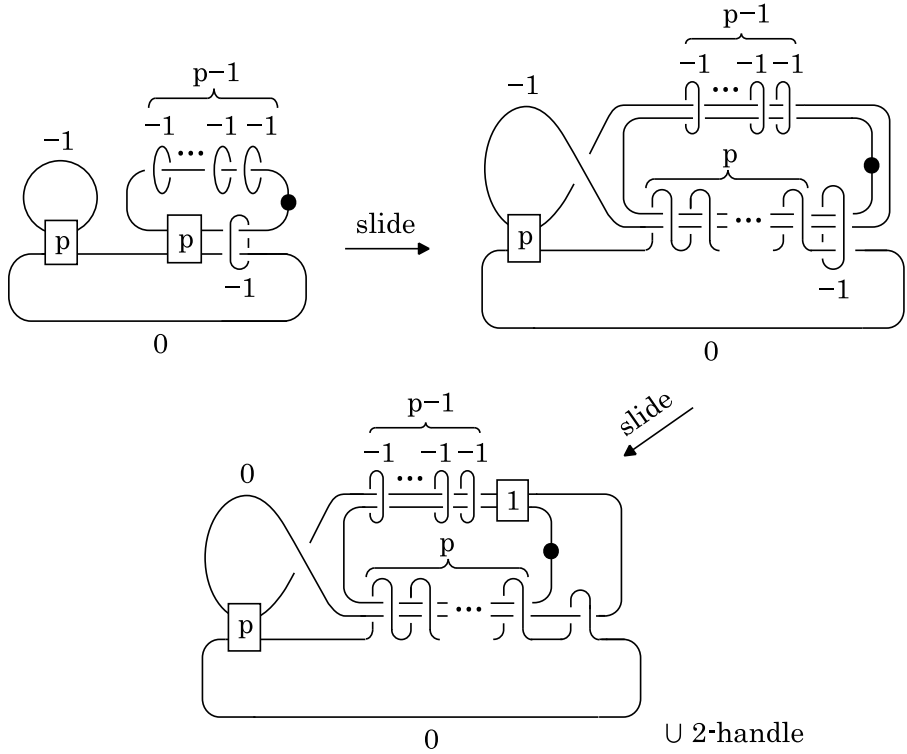}
\caption{handle slides of $D_p$}
\label{fig14}
\end{center}
\end{figure}
\begin{figure}[ht!]
\begin{center}
\includegraphics[width=4.95in]{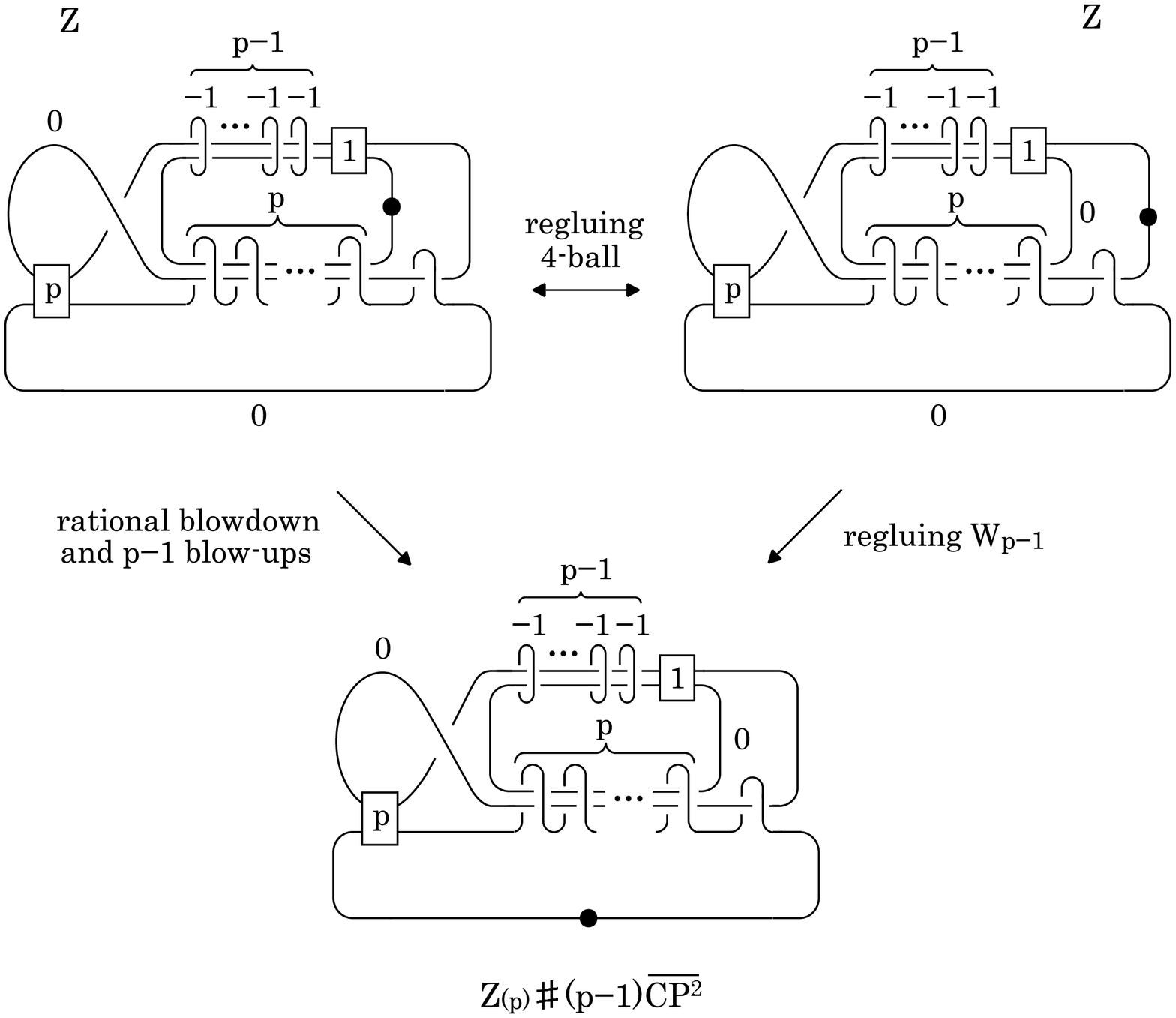}
\caption{}
\label{fig15}
\vspace{0.5in}
\end{center}
\end{figure}
\begin{figure}[ht!]
\begin{center}
\includegraphics[width=1.3in]{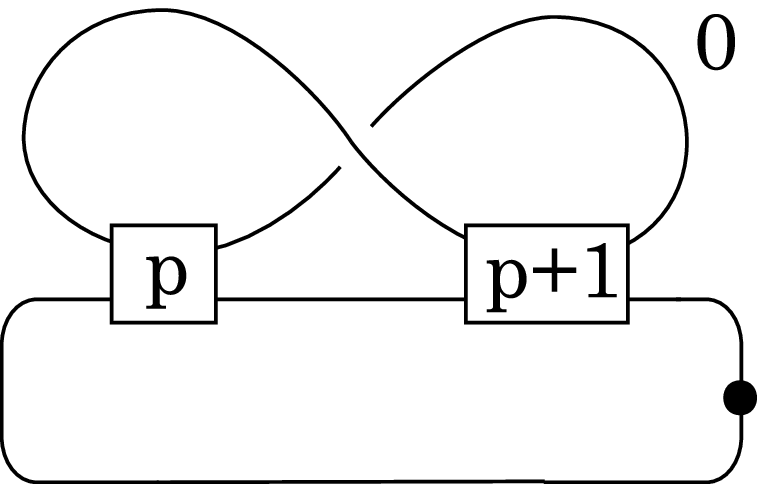}
\caption{$W_{p-1}$}
\label{fig16}
\end{center}
\end{figure}
\begin{figure}[ht!]
\begin{center}
\includegraphics[width=4.95in]{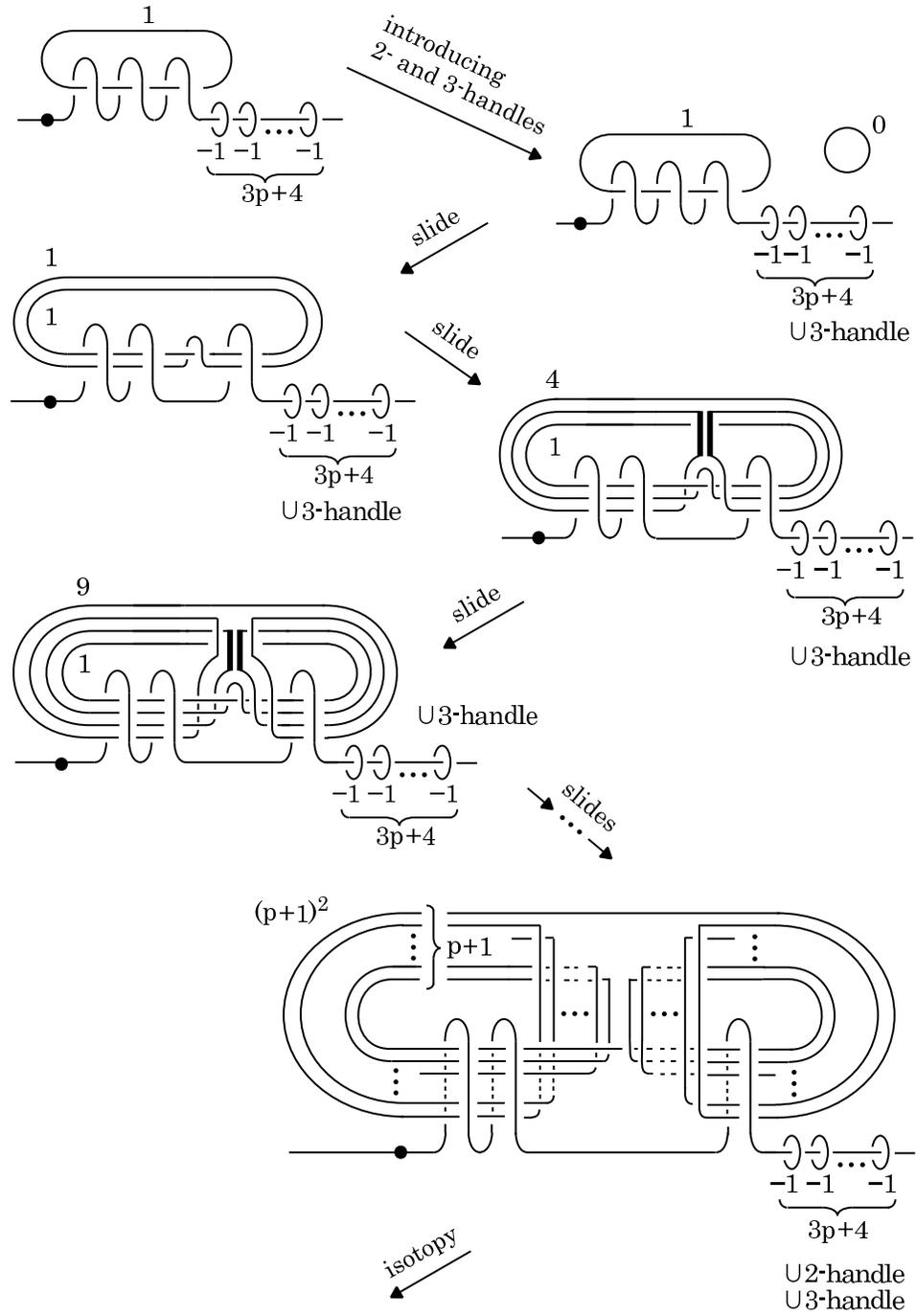}
\caption{introducing a $2$-handle/$3$-handle pair and sliding handles}
\label{fig17}
\end{center}
\end{figure}
\begin{figure}[ht!]
\begin{center}
\includegraphics[width=4.9in]{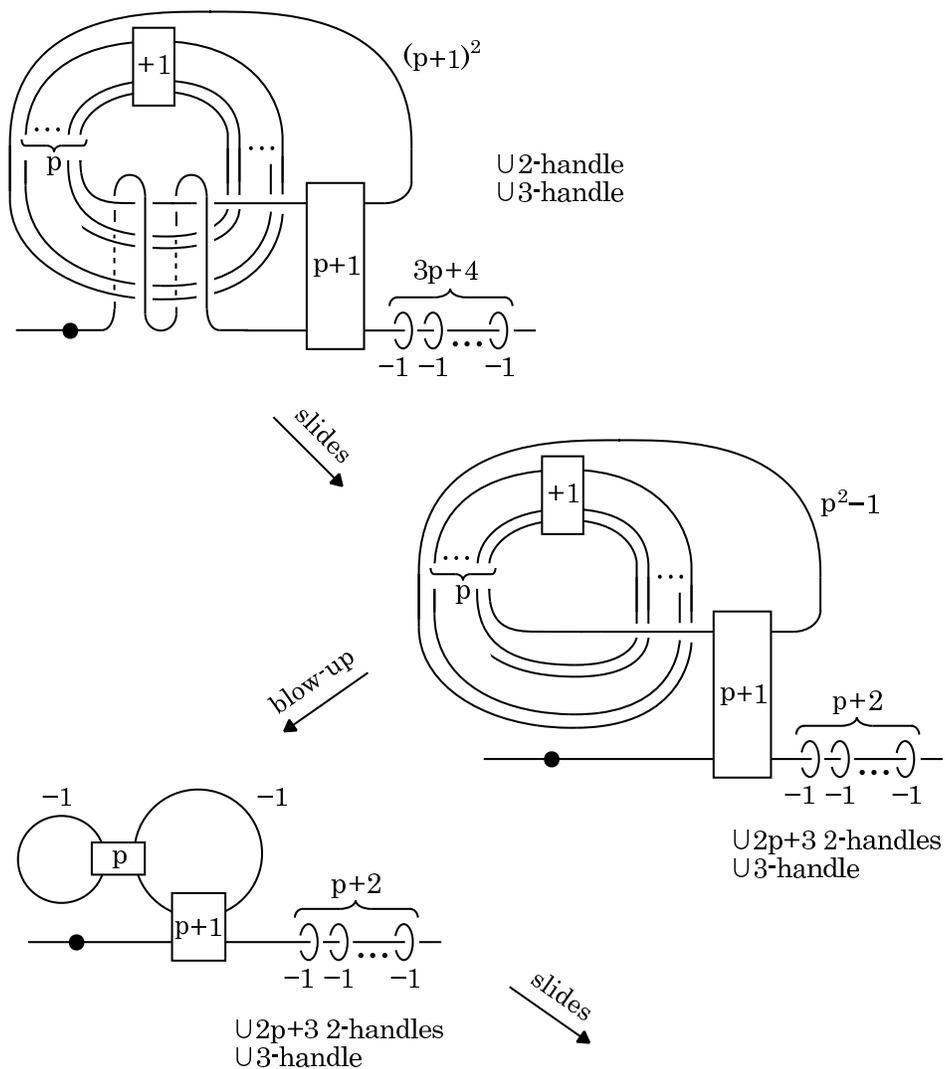}
\caption{handle slides and blow-up}
\label{fig18}
\end{center}
\end{figure}
\begin{figure}[ht!]
\begin{center}
\includegraphics[width=4.9in]{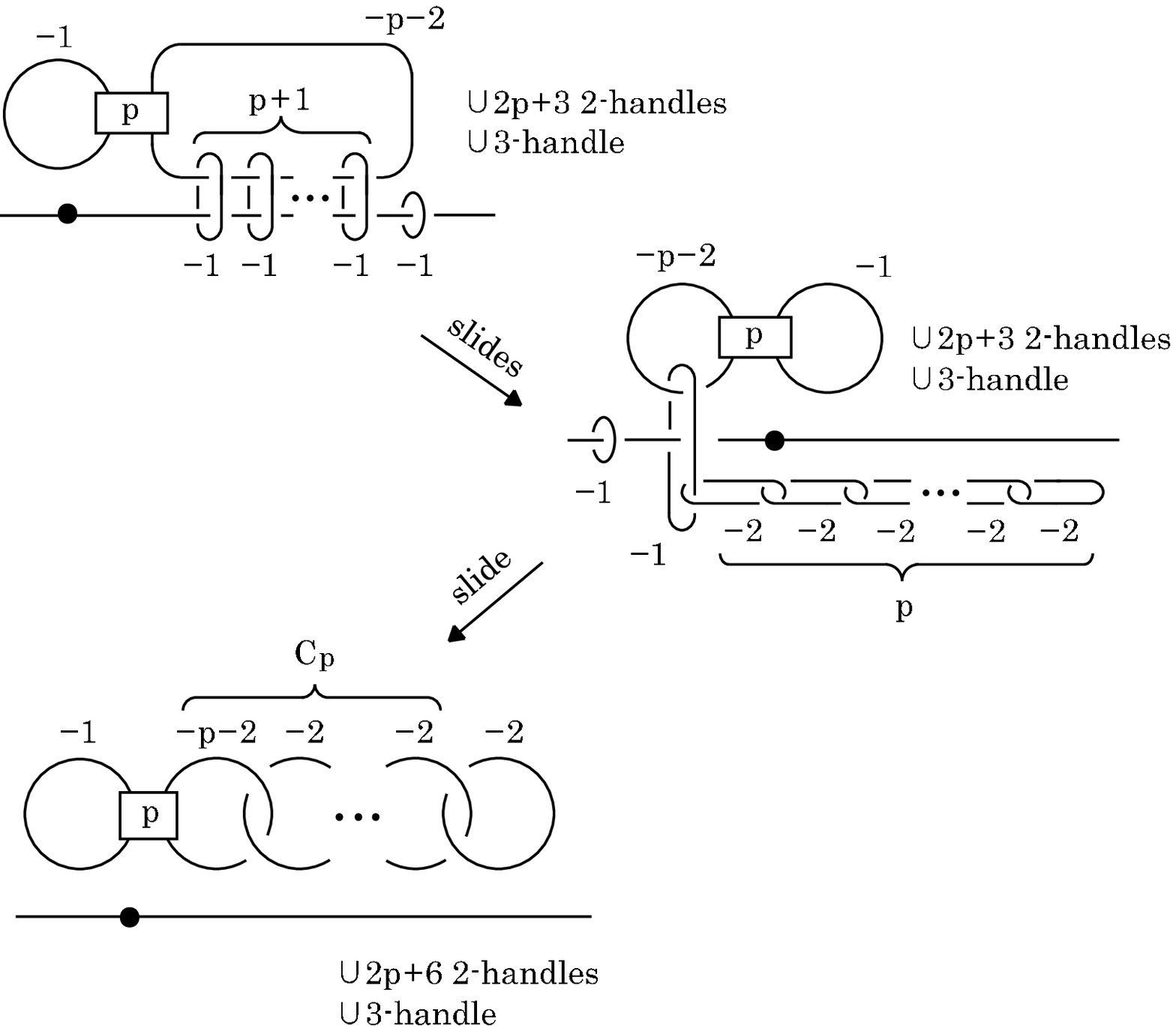}
\caption{handle slides}
\label{fig19}
\vspace{0.3\baselineskip }
\end{center}
\end{figure}

\begin{figure}[h!]
\begin{center}
\includegraphics[width=4.4in]{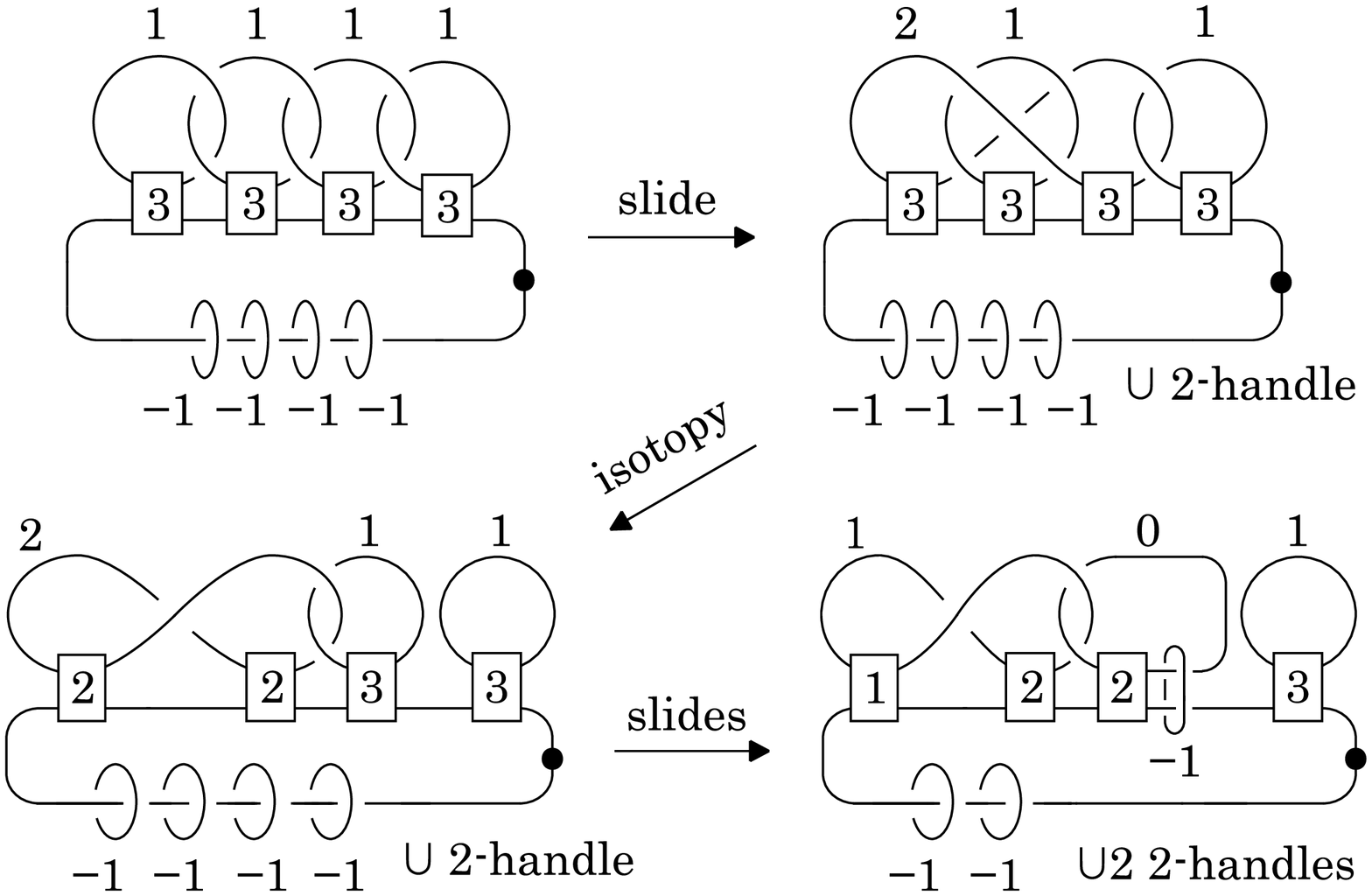}
\caption{handle slides}
\label{fig20}
\end{center}
\end{figure}
\clearpage

\end{document}